\documentclass{amsart}
\usepackage{subcaption}
\usepackage{graphicx}
\usepackage{amsmath}
\usepackage{amssymb}%
\usepackage{textcomp}
\usepackage{amsfonts}%
\usepackage{graphicx}
\usepackage[all]{xy}
\usepackage[mathscr]{euscript}
\usepackage{hyperref}
\newtheorem{theorem}{Theorem}[section]
\newtheorem{lemma}[theorem]{Lemma}
\theoremstyle{definition}
\newtheorem{definition}[theorem]{Definition}

\newtheorem{corollary}[theorem]{Corollary}

\usepackage[all]{xy}
\usepackage{graphicx}
\theoremstyle{remark}
\newtheorem{remark}[theorem]{Remark}
\numberwithin{equation}{section}
\usepackage{graphicx}

\begin{document}
 \title{Thin gordian Unlinks}
\author{Jos\'{e} Ayala}
\address{Universidad de Tarapacá, Casilla 7D, Iquique, Chile} 
\thanks{Supported by Fondecyt Grant \#11220579}
              \email{jayalhoff@gmail.com}


\baselineskip=20 true pt
\baselineskip=1.10\normalbaselineskip
\maketitle 
\begin{abstract} 
\baselineskip=20 true pt
\maketitle \baselineskip=1.10\normalbaselineskip
A gordian unlink is a finite number of unknots that are not topologically linked, each with prescribed length and thickness, and that cannot be disentangled into the trivial link by an isotopy preserving length and thickness throughout.

In this note, we provide the first examples of gordian unlinks. As a consequence, we identify the existence of isotopy classes of unknots that differ from those in classical knot theory. More generally, we present a one-parameter family of gordian unlinks with thickness ranging in $[1,2)$ and absolute curvature bounded by 1, concluding that thinner normal tubes lead to different rope geometries than those previously considered. Knots or links in the one-parameter model introduced here are called thin knots or links. When the thickness is equal to 2, we obtain the standard model for geometric knots, also called thick knots.

\end{abstract}

\section{Introduction} 

Informally, we may say that $n$ tangled unknots made out of flexible cord, none of them linked one to another, could be continuously deformed into $n$ untangled round circles, i.e. the trivial link of $n$ components. On the other hand, if the unknots have prescribed length and thickness, it is still an open problem to prove the existence of a tangled trivial link of $n$ components so that no isotopy preserving length and thickness can disentangle them into $n$ untangled round thick circles; we call these gordian unlinks. In this note, we provide a one-parameter family of these objects; see figures \ref{figfig5} and \ref{figfig6}.

M. Freedman, Z. He, and Z. Wang, while studying conformal invariance for the energies proposed by J. O'Hara \cite{ohara} suggested a candidate to be a gordian unknot \cite{freedman}. Later, Pieransky managed to numerically untangle the candidate given in \cite{freedman} by a computer program called SONO (shrink-on-no-overlaps) \cite{pieransky1, pieransky2}. 

The intuition behind calling a knot {\it thick} is due to (among others) characterisations given by R. Litherland, J. Simon, O. Durumeric, E. Rawdon in \cite{simon1}; Y. Diao, C. Ernst, E. J. Janse van Rensburg \cite{diao1} and the one by O. Gonzalez and J. Maddocks \cite{gonzalez}. A characterisation of thickness and equivalences between several characterisations of thickness have been provided by J. Cantarella, R. B. Kusner, and J. M. Sullivan \cite{cantarella 1}.

A central problem in geometric knot theory is the ropelength problem. This asks for the minimal ratio between the length of the core to the thickness over all the realisations in a knot or link type. There is a vast literature in non-classical knot theory; here we mention \cite{agol, diao1, hyde, durum, gonzalez, kaitrich, millet, simon1, ohara, heiko} to name just a few. To contextualise the difficulty of these problems, after decades of tremendous efforts, the only knot with exact known ropelength value is the unknot. For thick links, this situation is improved in \cite{cantarella 1}  by providing examples of families of chain links built from line segments and arcs of round circles.

 A. Coward and J. Hass proved the existence of a gordian link, that is, two unlinked thick knots that cannot be untangled by an isotopy that preserves length and thickness \cite{cowardhass}. This link corresponds to the connected sum of two trefoils split by an unknot. In addition, R. Kusner and W. Kusner exhibited a pair of links that are Gehring ropelength minimizers but are not isotopic while preserving geometric constraints \cite{wkusner}. More recently, J. Ayala and J. Hass, in Thick gordian Unlinks \cite{ayalahass}, addressed a similar question to the one studied here, but for knots with a thickness of 2. It is important to note that these are fundamentally different problems, as bounding the curvature while using thinner tubes leads to a geometric obstruction that is absent when the thickness is 2.

The standard formulation for the ropelength problem asks to minimise the length of a knot so that it remains of {\it unit thickness} (or at least 1) throughout. As we can see from ordinary experience, the large variety of shapes of wire ropes does not seem to have normalised thickness. In fact, the relation between the minimum bent radius of the core of a wire rope and its thickness strongly characterises the wire rope itself. As an experiment, take your headphone wire, bend it like a U-turn and push it through to see that there is a sort of turnbuckle at the pushed end; see figure \ref{figfigtubes}. Now, do the same with a thicker rope; most probably the end of this rope and the end of the headphones have different shapes. In this work, we restrict the conventional framework of geometric knot theory by considering the class of $C^1$-smooth and piecewise $C^2$ curves, included within the standard $C^{1,1}$-smooth curves (with the additional constraint on absolute curvature). $C^{1,1}$-smooth curves traditionally associated with well-established theories involving min-type Morse functions and the Kuhn–Tucker theorem on constrained optimisation, \cite{agol,cantarella 2}. Through these modifications, we aim to offer a more realistic model for physical knots, see figure \ref{figfigtubes}.

{ \begin{figure} 
 \begin{center}
\includegraphics[width=.6\textwidth,angle=0]{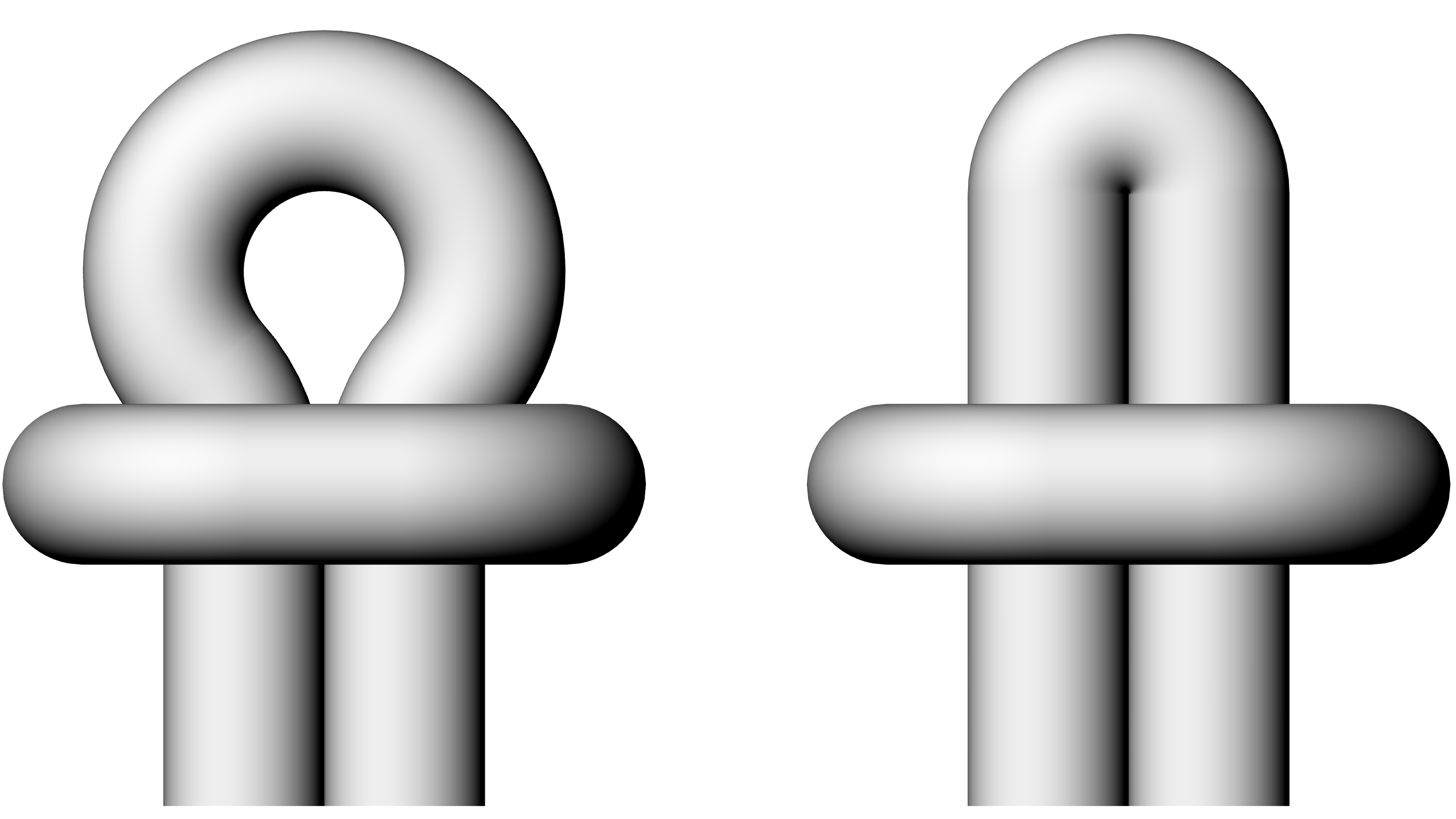}
\end{center}
\caption{Left: The end of a headphone wire is pushed with a tight float-like horizontal unknot to get stuck at some point. In this case, both the wire and the float have curvature bounded above by 1 and thickness 1. Right: This is the standard approach to geometric knots. Both the wire and the float have thickness 1, but different curvature bounded above by 2 and 1 respectively. In this case, the float can slide along the wire with no obstruction.}
\label{figfigtubes}
\end{figure}}

When minimising the length of a knot so that it remains of {\it unit thickness} (or at least 1) the normal map gets singular when the normals focus at the cap in figure \ref{figfigtubes} right. An important observation in this note is that {\it if the maximum radius of curvature does not match half the thickness of the rope, then there is no focusing of the normals at the ends of a rope}. This causes a defect for locking unknots. 

This observation led us to find minimisers in spaces of thin links (thickness ranging in [1, 2)) using results by Dubins and Sussmann, see \cite{dubins 1, sussman} together with results by the author and colleagues \cite{papera, papere, paperh} to conclude that thinner normal tubes lead to different rope geometry from the ones so far considered. Knots or links in the one-parameter model here introduced are called thin knots or links, and when thickness is equal to 2 we obtain the standard model considered for geometric knots, also called thick knots.

One of Hatcher's equivalences for the Smale conjecture states that the space of smooth unknotted loops in $\mathbb R^3$ deformation retracts onto the space of great circles in $S^2$ \cite{hatcher 1}. Recently, T. Brendle and A. Hatcher generalised this result for systems of unknots, none of them topologically linked one to the other \cite{hatcher 2}. They proved that the space of smooth links in $\mathbb R^3$ isotopic to the trivial link of $n$ components has the same homotopy type as the space consisting of configurations of $n$ unlinked unknots. The existence of gordian unlinks may indicate an obstruction for the existence of a version of the Brendle-Hatcher theorem for thin knots. 

In Section 2 we introduce basic notations, and show Lemma \ref{r1} which proves to be the foundation for the results coming in subsequent sections. In Section 3 we study the connected components in spaces of $\kappa$-constrained curves in $\mathbb R^3$. In Theorem \ref{kappacannotsing} we conclude that for sufficiently close initial and final points there is a connected component consisting exclusively of embedded curves. This leads to the existence of at least two connected components in spaces of $\kappa$-constrained curves, see Theorem \ref{ktrap}. In Section 4 we define separability for links and in Theorem \ref{exgords} we prove the existence of gordian unlinks. In Theorem \ref{mingords} we provide an explicit example of a gordian unlink. And, in Theorem \ref{ccccgordian} we show the existence of an infinite family of gordian unlinks.

\section{Geometric Obstructions}\label{lemmata}
We establish that a bound on curvature satisfied by the class of curves here studied leads to impediments to performing certain continuous deformations.

A curve $\gamma:[0,s]\to \mathbb R^3$ is said to be {\it in} $\mathcal S\subset \mathbb R^3$ if $\gamma(t) \in {\mathcal S}$ for all $t\in [0,s]$. Otherwise, the curve $\gamma$ is said to be {\it not in} $\mathcal S$. The interior, closure, boundary, diameter, and image of ${\mathcal S}$ under $\gamma$ are denoted by $int({\mathcal S})$, $cl({\mathcal S})$, $\partial{\mathcal S}$, ${diam}(\mathcal S)$, and $\gamma(\mathcal S)$ respectively. 

\begin{definition} \label{cbc} An embedded arc-length parameterised curve $\gamma: [0,s]\rightarrow {\mathbb R}^3$ is called {\it $\kappa$-constrained} if:
\begin{itemize}
\item $\gamma$ is $C^1$ and piecewise $C^2$
\item $||\gamma''(t)||\leq \kappa$, for all $t\in [0,s]$ when defined, $\kappa>0$ a constant.
\end{itemize}
 If $\gamma(0)=\gamma(s)$ then $\gamma$ is called a loop, otherwise $\gamma$ is called an arc. 
\end{definition}

Note that $\kappa$-constrained curves have absolute curvature bounded above almost everywhere by a positive constant with $1/\kappa$ corresponding to the minimum allowed radius of curvature. 

In spite of the generality carried by considering $C^{1,1}$ curves, we restrict ourselves to curves that are $C^1$ and piecewise $C^2$, these admit at most a finite number of points where the curvature is not defined. Note that all the known minimal ropelength links are indeed $C^1$ and piecewise $C^2$ \cite{cantarella 1}. In addition, the 2-dimensional counterpart of the ropelength problem, the ribbonlength problem, whose minimisers in spaces of immersed knots and link diagrams are $C^1$ and piecewise $C^2$ \cite{paperh}. 

Next, we prove that a $\kappa$-constrained arc in a 3-ball of radius $1/\kappa$ cannot intersect the boundary of the ball at an isolated point, see figure \ref{figspheres}. 

\begin{lemma}\label{3dntp} 
Suppose a $\kappa$-constrained arc $\gamma:[0,s]\to \mathbb R^3$ is defined in a radius $1/\kappa$ 3-ball $B$. Then, $\gamma([0,s]) \subset \partial B$ or $\gamma((0,s)) \cap \partial B=\emptyset$. 
\end{lemma}

\begin{proof} 
We prove that $\gamma$ while entirely defined in $B$ does not admit an isolated first order contact with $\partial B$. Suppose that an arc $\gamma$ defined in $B$ has an isolated contact point with $\partial B$ at $p=\gamma(t^*)$. Since $\gamma$ is of class $C^1$ this contact point is of first order. Therefore, the affine tangent plane $T_p \partial B$ is tangent to the osculating circle $O_p$. If $O_p \subset B$ it is proper, then the radius of $O_p$ is less than $1/\kappa$, implying that the curvature at $\gamma(t^*)$ is greater than $\kappa$, leading to a contradiction. If the radius of $O_p$ is $1/\kappa$, then by Definition \ref{cbc} it must be in $\partial B$. If the radius of $O_p$ is greater than $1/\kappa$, then $\gamma$ has a point near $p$ not in $B$, leading to a contradiction.
\end{proof}

\begin{definition}\label{leq2r1} 
 Let $B_1,B_2$ be radius $1/\kappa$ 3-balls, such that $\partial B_1\cap \partial B_2$ is a circle. Set
\begin{itemize}
\item  ${\mathcal I} ={int}( B_1\cap B_2)$
 \item  ${\mathcal U}=B_1\cup B_2$
\item ${\mathcal E}={ int}({\mathcal U}) \setminus { cl}({\mathcal I})$
\end{itemize}
\end{definition}


{ \begin{figure} 
 \begin{center}
\includegraphics[width=1\textwidth,angle=0]{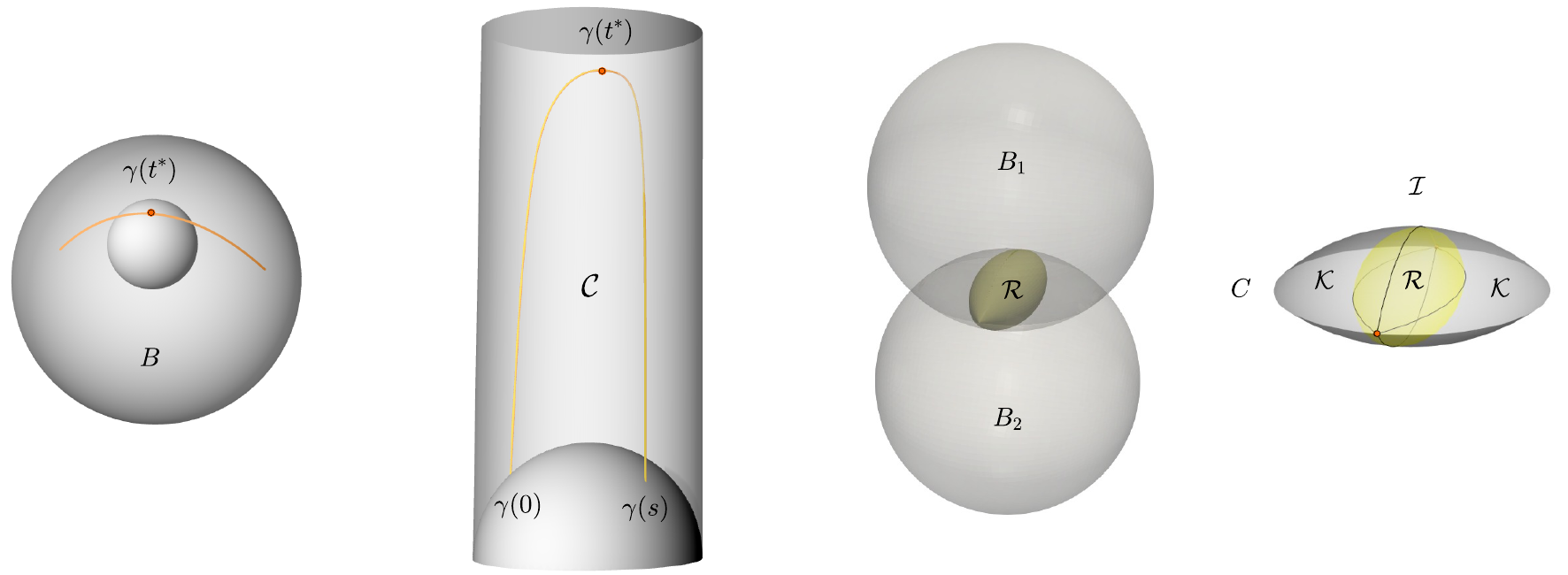}
\end{center}
\caption{From left to right. A first order contact of $\gamma$ with $\partial B$ leads to higher curvature. In this case, $\gamma$ cannot be entirely defined in $\mathcal C$. A generic intersection of two balls leading to $\mathcal R$. The intersection set $\mathcal I$ with $\mathcal K$ divided into two halves by $\mathcal R$.}
\label{figspheres}
\end{figure}}

\begin{lemma} \label{lemmae} 
 Let $S_1,S_2$ be radius $1/\kappa$ spheres, such that $S_1 \cap S_2=C$ is a circle. Then, a $\kappa$-constrained arc $\gamma:[0,s] \to \mathbb R^3$ with endpoints in $C$ such that $\gamma((0,s))\subset \mathcal E$ cannot exist.
\end{lemma}

\begin{proof}
Suppose such a curve exists. Let $S_i=\partial B_i$ where $B_i$ is a 3-ball, $i=1,2$, both intersecting at a circle $C$. Without loss of generality, suppose that $\gamma((0,s))\subset B_1$ having endpoints at $C$. Let $o$, the centre of $C$, be the origin, and let $h: \gamma([0,s])\to  \mathbb R$ be the projection map of $\gamma$ onto the $z$-axis. Since $\gamma([0,s])$ it is compact and $h$ continuous, then $h$ attains a maximum. Since $\gamma$ is of class $C^1$, the highest point $p=\gamma(t^*)$ attained by the curve has a contact of first order with a radius $1/\kappa$ sphere $S^*$ which is located below the affine tangent plane $T_p S^*$. By Lemma \ref{3dntp} such a contact point cannot exist, leading to a contradiction. 
\end{proof}

The following result underlines the existence of a geometric obstruction for continuous deformations of arcs. This obstruction depends on curvature, the distance between (fixed) endpoints, and leads to the existence of at least two distinct connected components in the space $\kappa$-constrained arcs in $\mathbb R^3$ connecting the endpoints, compare Theorem \ref{ktrap}. Informally, for endpoints distant apart less than $2/\kappa$ the line segment joining $\gamma(0)$ to $\gamma(s)$ and another arc joining the same endpoints but having a point above $B_2$ (or below $B_1$) are in distinct homotopy classes in the space of $\kappa$-constrained arcs in $\mathbb R^3$ connecting these endpoints. These ideas will be clarified in the next section. 

\begin{lemma}\label{r1}  
(Geometric obstruction). A $\kappa$-constrained arc $\gamma: [0,s]\to{\mathcal C}$ such that:
$${\mathcal C}=\{(x,y,z)\in{\mathbb R}^3\,|\, x^2+y^2<1/\kappa^2,\,z\geq 0 \}$$
is an open cylinder, cannot satisfy both:
\begin{enumerate}
\item $\gamma(0),\gamma(s)$ are points on the $xy$-plane.
\item If $S$ is a radius $1/\kappa$ sphere with centre on the negative $z$-axis, and $\gamma(0),\gamma(s)\in S$. Then, some point in the image of $\gamma$ lies above $S$.
\end{enumerate}
In addition, if $\gamma$ satisfies {\rm (2)} then its diameter is at least 2r.
\end{lemma}
\begin{proof} 
Suppose that exists a curve satisfying both items in the statement. Similarly to Lemma \ref{lemmae}, let $h: \gamma([0,s])\to  \mathbb R$ be the projection map of $\gamma$ onto the $z$-axis. The compactness of $\gamma([0,s])$ and the continuity of $h$ guarantee the existence of a maximum height $p=\gamma(t^*)$. By the continuity of $\gamma$, there exists $z_0<h(p)$ such that the plane $z=z_0$ intersects $\gamma$ in two points; these are distant apart less than $2/\kappa$. Let $C$ be the circle in the plane $z=z_0$ having these two points as antipodes. And, $S_1, S_2$ be two radius $1/\kappa$ spheres whose intersection corresponds to $C$. By  Lemma \ref{lemmae} this curve cannot exist, leading to a contradiction.

If an arc satisfying (2) has diameter less than $2/\kappa$. Then, its projection onto the $xy$-plane must be a subset of the disk $x^2+y^2<1/\kappa^2$ implying that the arc is in $\mathcal C$, leading to a contradiction.
\end{proof}

\section{Spaces of $\kappa$-constrained arcs} \label{arcs}

If an arc is continuously deformed under the parameter $p$, we reparametrise each of the deformed arcs by its arc length. In this fashion, $\gamma_p: [0,s_p]\rightarrow {\mathbb R}^3$ describes a deformed arc at parameter $p$, with $s_p$ corresponding to its arc-length. We abuse notation by referring to a homotopy of arcs as such a reparametrisation by guaranteeing that each arc in the homotopy is parametrised by arc length. The space of $\kappa$-constrained arcs connecting $x,y \in \mathbb R^3$ is considered with the $C^1$ metric, and it is denoted by $\Sigma(x,y)$. 

\begin{definition}  \label{hom_adm} Assume $\beta,\gamma \in \Sigma(x,y)$ are given. A {\it $\kappa$-constrained homotopy} between $\beta: [0,s_0] \rightarrow \mathbb R^2$ and $\gamma: [0,s_1] \rightarrow \mathbb R^2$ corresponds to a continuous one-parameter family of embedded paths $ {H}_t: [0,1] \rightarrow \Sigma(x,y)$ such that:
\begin{itemize}
\item ${H}_t(p): [0,s_p] \rightarrow \mathbb R^2$ for $t\in [0,s_p]$ is an element of $\Sigma(x,y)$ for all $p\in [0,1]$.
\item $ {H}_t(0)=\beta(t)$ for $t\in [0,s_0]$ and ${H}_t(1)=\gamma(t)$ for $t\in [0,s_1]$.
\end{itemize}
\end{definition}

\begin{remark} Bounded curvature homotopy between bounded curvature paths is an equivalence relation, which will be denoted by $\sim$.

A {\it homotopy class} in $ \Sigma(x,y)$ is an equivalence class in $ \Sigma(x,y)/\sim$.

 Such a {\it homotopy class} is a path component in $ \Sigma(x,y)$.

\end{remark}

\begin{definition}\label{leq2r3} 
 Let $B_1,B_2$ be radius $1/\kappa$ 3-balls, such that $\partial B_1\cap \partial B_2$ is a circle $C$.  
\begin{itemize}
\item Let $x,y$ be an antipodal points in $C$. Let $\mathcal R$ be the open region enclosed by the union of all the short arcs in the family of radius $r$ circles passing through $x$ and $y$. 
\item $\mathcal K=\mathcal I \setminus cl (\mathcal R)$
\end{itemize}
see figure \ref{figspheres}.
\end{definition}

\begin{theorem}\label{below} An arc $\gamma \in \Sigma(x,y)$ such that $0<||x-y||<2/\kappa$ and such that $\gamma((0,s))\subset \mathcal S$, where $\mathcal S$ is $\mathcal K,\, \mathcal E,\,\mbox{or}\,\, \mathcal E \cup \mathcal K$ cannot exist. 
\end{theorem}

\begin{proof} Suppose that such an arc $\gamma$ in $\mathcal S=\mathcal K$ exists. Let $C$ be the circle in $B_1\cap B_2$ and let $C^*$ be one of the two short arcs of a radius $1/\kappa$ circle passing through $x$ and $y$ lying in the disk bounded by $C$, see figure \ref{figspheres} right. Let $S$ be the radius $1/\kappa$ sphere including $C^*$. Since $\gamma((0,s))\subset \mathcal K$, then it has a point above $S$. By Lemma \ref{r1} we have that $diam(\gamma((0,s)))\geq 2/\kappa$. On the other hand, since $S_1 \cap S_2\neq \emptyset$ we have that $diam(\mathcal I)<2/\kappa$. Since, $\mathcal K\subset \mathcal I$, we obtain a contradiction. The case $\mathcal S=\mathcal E$ is ruled in Lemma \ref{lemmae}. And, the case $\mathcal S=\mathcal E \cup \mathcal K$ is ruled similarly as before.
\end{proof}

\begin{corollary}\label{trappedcbc} An open arc $\gamma:(0,s)\rightarrow cl (\mathcal R)$ does not admit a first-order contact point with $\partial{\mathcal R}$.
\end{corollary}
\begin{proof}  After supposing that $\gamma$ has a first-order contact point with $\partial{\mathcal R}$, the result follows immediately from Lemma \ref{3dntp}.
\end{proof}

\begin{theorem}\label{c} Any arc in $\partial{\mathcal R}$ is not $\kappa$-constrained homotopic to an arc not in $cl(\mathcal R)$. 
\end{theorem}

\begin{proof} Let $\beta \in \Sigma(x,y)$ be an arc in $\partial{\mathcal R}$ of length $l$. 
Suppose there exists a $\kappa$-constrained homotopy $ {H}: [0,1] \rightarrow \Sigma(x,y)$ such that $H(0)=\beta$ and $H(1)=\gamma$ where $\gamma$ is not it $cl({\mathcal R})$. Since homotopies are continuous, there exist $\epsilon,\epsilon^*>0$ small, and an arc $\eta$ with $\eta \neq \beta$ such that $H(\epsilon)=\eta$ being $\eta$ a curve of length $l+\epsilon^*$. Accordingly, consider an open radius $1/\kappa$ cylinder ${\mathcal C}$ whose base disk has centre $o$ as in Lemma \ref{lemmae}. By Lemma \ref{r1}, $\eta$ cannot exist, leading to a contradiction.
\end{proof}

\begin{theorem}\label{kappacannotsing} The set $cl({\mathcal R})$ encloses only embedded arcs. 
\end{theorem}

\begin{proof} Suppose $cl({\mathcal R})$ encloses only embedded arcs. Let $\beta$ be embedded in $cl({\mathcal R})$. Let $ {H}: [0,1] \rightarrow \Sigma(x,y)$ be a $\kappa$-constrained homotopy such that $H(0)=\beta$ and $H(1)=\gamma^*$ are in $cl({\mathcal R})$ and admit a self-intersection. Since homotopies are continuous maps, there exists $t_0\in [0,1]$ such that ${ H}(t_0)=\gamma$ is the first arc in ${H}$ with a self-intersection. Let $\gamma(t_p)=\gamma(t_q)$ be such a self-intersection point, $t_p,t_q \in [0,s_\gamma]$. 

We apply Lemma \ref{r1} to the restriction $\gamma: [t_p,t_q] \to cl({\mathcal R})$ to conclude that the $diam(\gamma)\geq 2/\kappa$. On the other hand, $diam(cl({\mathcal R}))<2/\kappa$, so we conclude that such an arc is not in $cl({\mathcal R})$ leading to a contradiction.
\end{proof}

\begin{definition} \label{trapreg} A compact, connected, minimal with respect to set inclusion set $\mathcal S(x,y) \subset \mathbb R^3$ is said to be a {\it trapped} region if the arcs in $\Sigma(x,y)$ defined in $\mathcal S(x,y)$ are embedded and are not $\kappa$-constrained homotopic to arcs not in $\mathcal S(x,y)$.
\end{definition}

\begin{theorem}\label{ktrap} If $x,y\in {\mathbb R^3}$ are such that $0<||x-y||<2/\kappa$. Then, the space of arcs in $cl({\mathcal R})$ and the space of arcs not in $cl({\mathcal R})$ are in different homotopy classes  in $\Sigma(\mbox{\it x,y})$. In particular, $\mathcal S(x,y)=cl(\mathcal R)$ is a trapped region. 
 \end{theorem}

\begin{proof} The result follows after combining Corollary \ref{trappedcbc}, Theorem \ref{c} and Theorem \ref{kappacannotsing}, we leave the details to the reader.
\end{proof}

%
%

Let $\mathcal V=int(\mathcal U)\cup\{x,y\}$. Observe that $\mathcal V$ is also a trapping region. In fact, both $\mathcal I$ and $\mathcal V$ contain the arcs in $cl({\mathcal R})$. However, neither $\mathcal V$ nor $\mathcal I$ is minimal with respect to set inclusion.

We conjecture that for $0<||x-y||<2/\kappa$ the arcs in $cl({\mathcal R})$ and the arcs not in $cl({\mathcal R})$ are the only two homotopy classes of arcs in $\Sigma(\mbox{\it x,y})$. On the other hand, if $||x-y||\geq 2/\kappa$ or, $x=y$ we conjecture that $\Sigma(\mbox{\it x,y})$ consists of a single homotopy class.

\section{Existence of gordian unlinks} 

We consider a knot to be the image of an embedding of a $\kappa$-constrained curve in $\mathbb R^3$ that is homeomorphic to $S^1$. A $\kappa$-constrained isotopy is a $\kappa$-constrained homotopy through embeddings. 

In the context of classical knot theory, the following definition is trivially satisfied by any pair of unlinked knots. We abuse notation and refer to $\gamma$ as the knot itself.

\begin{definition}\label{sep} The knots $\beta,\gamma$ are said to be separable if there exist $1$-constrained isotopies $H$ of $\beta$, and $J$ of $\gamma$ and arbitrarily large neighbourhoods $M$ of $H(1)$ and $N$ of $J(1)$ such that: $$M(H(1))\cap N(J(1))=\emptyset $$
If $\beta,\gamma$ are not separable, they are said to be gordian. 
\end{definition}

With Lemma \ref{r1} and Theorem \ref{ktrap} in mind, we present the following. 

 \begin{definition} \label{shortlong}Let $B$ be a radius $1/\kappa$ 3-ball centred at the $z$-axis. A short arc has its endpoints on $\partial B$ and it is defined entirely on $\partial B$, or it is defined in $\mathcal R \subset int(B)$ except at its endpoints. A long arc has its endpoints on $\partial B$ and has a point above $S$.
\end{definition}

Note that 
\begin{itemize}
\item a short arc satisfies (1) in Lemma \ref{r1} and
\item a long arc satisfies (2) in Lemma \ref{r1}.
\end{itemize}


 \begin{corollary} \label{nonsplit} 
Consider,
\begin{itemize}
\item $\gamma$ to be a circle of radius $r_0<1/\kappa$ centred at the origin and located in the $xy$-plane
\item $\beta$ to be a $\kappa$-constrained unknotted loop traversing the interior of the disk  bounded by $\gamma$ with a long arc defined above and below the $xy$-plane.
\end{itemize} 
Then, $\beta$ and $\gamma$ are not separable, figure \ref{figstadium} left. 
\end{corollary} 

\begin{proof} 
Let $D$ be the disk bounded by $\gamma$ in the $xy$-plane. If $\beta$ and $\gamma$ are separable, then there exists a $\kappa$-constrained homotopy $H$ such that $H(0)=\beta$ and $H(t^*)\cap D=\emptyset$ for some $t^*\in [0,1]$. Without loss of generality, suppose that the homotopy $H$ pulls $\beta$ so it passes through $int(D)$. 

Since $\beta$ admits a long arc above the $xy$-plane, it satisfies (2) in Lemma \ref{r1}. We claim that every loop in the family $H(t)$, $t\in [0,1]$ must have a point above the $xy$-plane. Otherwise, if $H(t)$ passes through $int(D)$, by continuity, there exists $\hat t \in [0,1]$ such that $H(\hat t)$ has a long arc above the $xy$-plane while at the same time is inside an open cylinder with base $D$, contradicting Lemma \ref{r1}. We conclude that $H(t)\cap D\neq\emptyset$ for all $t\in[0,1]$, and therefore such a homotopy cannot exist. Since the non-existence of a homotopy implies the non-existence of an isotopy, the result follows.
\end{proof}

%

\begin{corollary} \label{nonsplitknot} 
Suppose $\beta$ in Corollary \ref{nonsplit} is any knot. Then $\beta$ and $\gamma$ are not separable.   
\end{corollary} 
\begin{proof} Identical to Corollary \ref{nonsplit}. 
\end{proof}

{ \begin{figure} 
 \begin{center}
\includegraphics[width=.8\textwidth,angle=0]{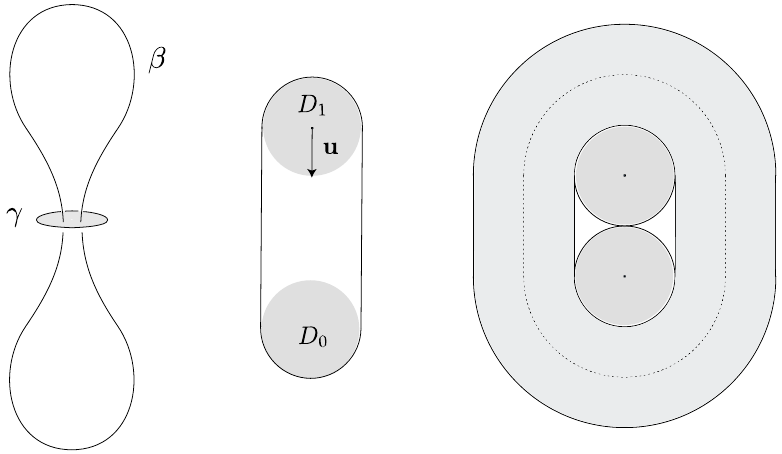}
\end{center}
\caption{Left: Examples of unknots as in Corollary \ref{nonsplit}. Centre: An illustration of the gradient descent method applied to disks. Right: The resulting minimal ribbon diagram in Lemma \ref{ribbon3d}.}
\label{figstadium}
\end{figure}}

\subsection{Curvature and thickness}
We adopt the notion of radius of thickness given in \cite{simon1} as our starting point, see also \cite{diao1}. Let $R_1(\gamma)=1/\max \kappa(s)$; this quantity corresponds to the minimum radius of curvature allowed by $\gamma$, so in our case $R_1(\gamma)=1/\kappa$. In addition, we say that $p,q \in \gamma$ are double critical points if the cord connecting them is orthogonal to the tangent vectors of $\gamma$ at $p$ and $q$. The minimum amongst the distances between all doubly critical points in $\gamma$ is denoted by $R_2(\gamma)$. 

\begin{theorem}(\cite{simon1}). \label{simon}
 The radius of thickness of a knot $\gamma$ is given by, $$\tau(\gamma) =  min\{R_1(\gamma),\frac{1}{2} R_2(\gamma) \}.$$
\end{theorem}

The definition of radius of thickness was originally presented for $C^2$ knots, and it remains valid for knots that are $C^1$ piecewise $C^2$.

\begin{definition} The immersed open tubular neighbourhood of a $\kappa$-constrained knot or link with thickness ranging in $[1,2)$ and fixed length is called a thin knot or link. Two links are thin isotopic if they are $\kappa$-constrained isotopic while preserving the thickness and the length of each of its components. Knots or links that are isotopic but not thin isotopic are called gordian. \end{definition}

The class of equivalence of thin links arises from the notion of \emph{thin isotopy}, which imposes stricter constraints compared to the standard isotopy. A thin isotopy requires that the deformation occurs exclusively between $\kappa$-constrained links, ensuring that the curvature constraint is preserved throughout. In addition, the thickness of each component is the same and remains fixed within the range $[1,2)$, and the length of each component is also maintained.

Theorem \ref{ktrap} can be viewed as an obstruction for local deformations of arcs in the core of a thin knot, while keeping fixed its tangent vectors. Statements with a similar flavour can be found in \cite{heiko, gonzalez}.

\begin{corollary} A short arc in the core of a thin knot cannot be deformed outside $cl(\mathcal R)$ while keeping the endpoints and directions fixed. In addition, under these hypotheses, all admissible deformations remain embeddings.
\end{corollary}

\begin{proof} Fix the tangent directions of an arc. The locus of possible deformations for such an arc must be a subset of $cl(\mathcal R)$. Theorem \ref{kappacannotsing} guarantees embeddedness and Theorem \ref{ktrap} guarantees these must be confined in $cl(\mathcal R)$.
\end{proof}

A flat ribbon knot is an immersion of an annulus into the Euclidean plane such that the core of the annulus corresponds to an immersed knot diagram. The ribbonlength problem aims to find the minimal ratio between the length of the core to the ribbon width over all the planar realisations in a knot or link type,
$$\mbox{Rib}(\gamma)= \frac{\mbox{Length}(\gamma)} {\mbox{Width}(\gamma)} .$$  


The next result is used to prove the minimality of a planar stadium curve enclosing two disjoint radius $1/\kappa$ disks (this corresponds to the core of a thin unknot). This result can also be proved using the techniques in \cite{cantarella 1}.

\begin{lemma} \label{ribbon3d}
 Consider the plane minus the interior of two disjoint radius $1/\kappa$ disks. The minimal ribbonlength diagram enclosing both disks is achieved by the unknot with core a stadium curve of ribbonlength $2 (\pi+1)$. 
 \end{lemma}

\begin{proof} 
Minimal length ribbon diagrams exist by Theorem 3.3 in \cite{paperh}. We apply the gradient descent method in Section 5 in \cite{paperh} to a generic diagram enclosing the two disks, see Fig. \ref{figstadium}. It is not hard to see that such a diagram must be embedded. Since the disks are disjoint, a minimal layout of these must intersect its boundaries at a single point. The descent leads to a minimal ribbon diagram being a stadium curve of length $4(\pi+1)/\kappa$ and width $2/\kappa$, and therefore of ribbon length $2 (\pi+1).$. Note that any perturbation of the core will lead to a length increase, concluding the proof. 
\end{proof}

Observe that the next result proves the existence of a gordian unlink without imposing a condition on the length of, $\gamma$ besides the existence of 2 long arcs.

\begin{theorem}\label{exgords} 
There exists an infinite number of gordian unlinks.
\end{theorem}

\begin{proof} 
Consider a thin unknot with core $\gamma$ traversing the $xy$-plane orthogonally at $\gamma(t_1),\gamma(t_2)$. Suppose that the portions of $\gamma$ defined above and below the $xy$-plane are both long arcs. Due to the thickness constraint, we have that $||\gamma(t_1)-\gamma(t_2)||\geq2 \tau(\gamma)$. 

Consider the relation $\kappa=2\tau(\gamma)$ between the curvature bound and thickness of $\gamma$. Suppose that $\tau(\gamma)=1/2$ so that the $\min \{R_1(\gamma),\frac{1}{2} R_2(\gamma) \}$ is achieved at $\frac{1}{2} R_2(\gamma)$. We conclude that $R_2(\gamma)=1$ and that $R_1(\gamma)=1/\max \kappa(s)=1/\kappa=1$. We set $\gamma$ to satisfy $||\gamma_1(t_1)-\gamma_1(t_2)||=1$ in the $xy$-plane.

By Lemma \ref{ribbon3d} the tightest core diagram enclosing two radius $1/2$ disks in the $xy$-plane has core a $1$-constrained stadium curve denoted now as $\beta$, see figure \ref{figstadium} right. Consider the thin unknot with core $\beta$. Note that any perturbation of $\beta$ would increase its length. Note also that $\gamma, \beta$ have both $1$-constrained cores and thickness $1$.
 
Let $D$ be the unit radius disk enclosing the radius $1/2$ disks $D_1,D_2$ corresponding to the intersection of the thin unknot with core $\gamma$ and the $xy$-plane. Note in particular that we have $\gamma(t_1),\gamma(t_2) \in int(D)$.

If $\gamma, \beta$ are separable, then there exists a $\kappa$-constrained isotopy $H$ such that $H(0)=\gamma$ and $H(t^*)\cap  D=\emptyset$ for some $t^*\in [0,1]$. Since $\gamma$ has a long arc above and below the $xy$-plane via Lemma \ref{r1} we have that $H(t)\cap\  D\neq\emptyset$ for all $t\in[0,1]$, and therefore such an isotopy cannot exist, concluding by Definition \ref{sep} that $\beta$ and $\gamma$ are the cores of a pair of gordian unlinks. 

Since the only constraint on the length of $\gamma$ is the existence of a long arc below and above the $xy$-plane (which constitutes a lower bound on length) we can successively increase the length of $\gamma$ and apply Lemma \ref{r1} to find an infinite number of gordian unlinks.
\end{proof}

\subsection{Critical ropelength gordian unlinks}

A planar bounded curvature path corresponds to a $C^1$ and piecewise $C^2$ path in $\mathbb R^2$. These paths connect two {\it fixed} elements in the tangent bundle $T\mathbb R^2$ and have curvature bounded by a positive constant $\kappa$.  In 1957 L. E. Dubins \cite{dubins 1}  characterised the minimal length planar bounded curvature paths motivated by questions of A. Markov \cite{markov}. These correspond to paths of type {\sc csc} or  {\sc ccc} where {\sc c} stands for an arc of a circle of radius $1/\kappa$ and {\sc s} a line segment. 

A bounded curvature path connecting fixed $(x,X), (y,Y) \in T{\mathbb R}^3$ starts at $x$ tangent to $X$ ending at $y$ tangent to $Y$ with absolute curvature at most $\kappa>0$ throughout. We consider a result of H. Sussmann in which he characterised the minimal length bounded curvature paths in $\mathbb R^3$ \cite{sussman}. Besides the {\sc csc}-{\sc ccc} characterisation, he found the existence of helicoidal arcs, characterised by the fact that their torsion satisfies a second order ordinary differential equation. We present the following result in its original form for $\kappa=1$.

\begin{theorem} \label{mds}(Markov-Dubins-Sussman). Choose $(x,X), (y,Y) \in T{\mathbb R}^3$. A length minimising bounded curvature path from $(x,X)$ to $(y,Y)$ is either a {\sc ccc} path having its middle component of length greater than $\pi $ or a {\sc csc} path. Some of the circular arcs or line segments may have zero length. Or, a helicoidal arc where torsion $\alpha$ never vanishes and satisfies the following differential equation,
$$\alpha''=\frac{3 \alpha'^2}{2 \alpha}- 2\alpha^3+2 \alpha-\xi \alpha |\alpha|^{\frac{1}{2}}, \qquad \xi \geq0.$$
\end{theorem}

Computing length minimising planar bounded curvature paths is a delicate issue. For example, in many cases, the length variation between length minimisers of arbitrarily close endpoints or directions is discontinuous. In addition, the symmetry property metrics satisfy is in general violated. That is, in general, the length of the Dubins path from $(x,X)$ to $(y,Y)$ is different from $(y,Y)$ to $(x,X)$. Length minimising bounded curvature paths may not be unique, refer to \cite{paperf} for details about these claims.

\begin{definition} The ropelength $\mbox{Rop}(\gamma)$ of a thin knot $\gamma$ is the ratio of the length of the core to its thickness:
$$\mbox{Rop}(\gamma)= \frac{\mbox{Length}(\gamma)} {2 \tau(\gamma)}$$
\end{definition}

The ropelength problem asks for the minimal ropelength over all the realisations in a knot or link type. The ropelength problem is the 3-dimensional version of the ribbonlength problem. 

In Theorem \ref{exgords}, we constructed a family of gordian unlinks, where the core of one component, denoted $\gamma$, traverses orthogonally through the $xy$-plane, allowing a long arc both above and below the plane without any additional condition on its length. The other component, denoted $\beta$, has its core defined by a stadium curve.


We would like to highlight the following special case where both cores curves are $1$-constrained with thickness $1$. 

\begin{theorem}\label{mingords} There exists a minimal ropelength gordian unlink.
\end{theorem}

\begin{proof} 
 We first search for the minimal length $1$-constrained path starting at $x=(0,0,0)$ tangent to $X=(0,0,1)$ finishing at $y=(-1,0,0)$ tangent to $Y=(0,0,-1)$ based at $y$. Since the initial and final points and directions are coplanar, according to Theorem \ref{mds}, the solution has torsion zero and must be {\sc csc} or {\sc ccc}. We search amongst the six possible solutions given by the {\sc csc}-{\sc ccc} types in the $xz$-plane. By considering travel orientation (left-right) we conclude that the solution is a {\sc ccc} path of type left-right-left. Now, consider the solution path from $(y,Y)$ to $(x,X)$. This solution is also {\sc ccc} path. And more interestingly, the pair of {\sc ccc} solutions are mirror symmetric with respect to the $x$-axis in the $xz$-plane, see figure \ref{figccccfam}. 

Let $\gamma$ be the concatenation of the two {\sc ccc} paths. Clearly, $\gamma$ has a long arc above and below the $xy$-plane. Therefore, by Corollary \ref{nonsplit} $\gamma$ and the stadium unknot constructed in Theorem \ref{exgords} are non-separable unknots, and therefore gordian unknots. 

Since both {\sc ccc} paths are of minimal length under the given constraints. And, due to symmetry, $\gamma$ is also of minimal length under these constraints. We use the methods in \cite{paperf} to compute the following values. Since the length of each {\sc ccc} path is approximately $6,0325$ we have that $Length(\gamma)\approx 2\times 6,0325$ and $Length(\beta)\approx 8,2831$ both of thickness $1$ we conclude that:
$$\mbox{Rop}(\gamma \cup \beta)\approx 20.3481.$$

\end{proof}

%


{ \begin{figure} 
 \begin{center}
\includegraphics[width=1\textwidth,angle=0]{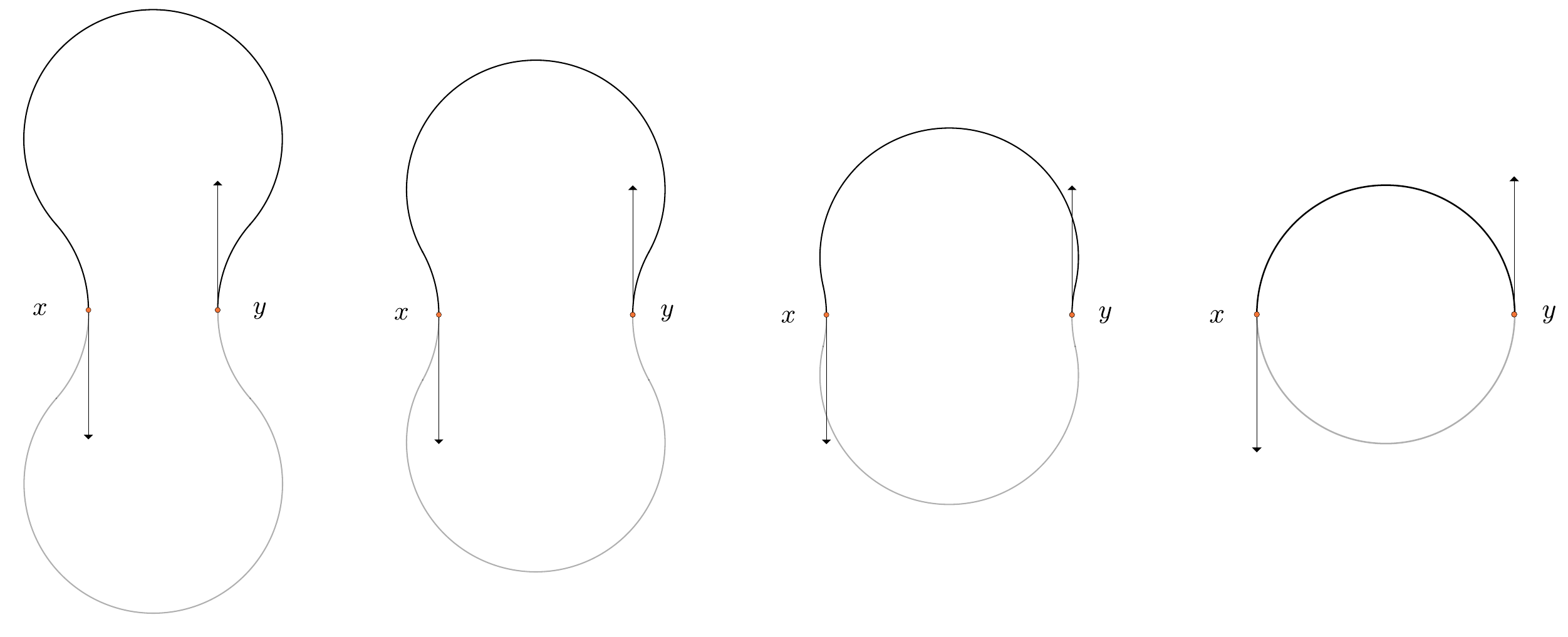}
\end{center}
\caption {Some solutions for the 2-dimensional Dubins problem in the $xz$-plane. In grey, we depict the {\sc ccc} solutions from $(x,X)$ to $(y,Y)$ and in black, solutions from $(y,Y)$ to $(x,X)$. When concatenated, these become elements of a family of {\sc cccc} unknots. The last picture is a radius 1 circle, and it corresponds to the limit case when of thickness $2\tau=2$. Note these are $1$-constrained curves.}
\label{figccccfam}
\end{figure}}

Next, we generalize the construction of the gordian unlink in Theorem \ref{mingords} and prove the existence of a family of gordian unlinks with thickness parametrized in the interval $[1, 2)$. In this case, both unknots have piecewise constant curvature and different thicknesses.


\begin{theorem}\label{ccccgordian} 
There exists a one-parameter family of minimal ropelength gordian unlinks. 
\end{theorem}

\begin{proof}  We construct a family of gordian unlinks whose elements are two component thin unknots $\{\gamma_\tau, \beta_\tau\}$. The cores $\gamma_\tau$ are $1$-constrained curves with thickness $2\tau$, and they traverse orthogonally the $xy$-plane while admitting a long arc above and below the $xy$-plane. The cores $\beta_\tau$ are also $1$-constrained stadium curves $\tau \in (1/2,1)$. 

We look for minimal length $1$-constrained paths connecting $(x,X_\tau), (y_\tau,Y_\tau)\in T\mathbb R^3$ where $X_\tau=(0,0,1)$ is based at $x=(0,0,0)$ and $Y_\tau=(0,0,-1)$ is based at $y_\tau\in (-\delta,0,0)$ for $\delta \in (1,2)$. We set $\tau=\frac{||x-y_\tau||}{2}$. We also search for Dubins minimal length path from $(y_\tau,Y_\tau)$ to $(x,X_\tau)$. 

As in Theorem \ref{mingords} the initial and final points and directions are coplanar, so both solutions have torsion zero and must be of {\sc ccc} type. Note that each $\gamma_\tau$ is a loop of type {\sc cccc} of minimal length. 

 Similar to Theorem \ref{mingords}, but for each, $\gamma_\tau$ we consider the minimal ribbon diagram $\beta_\tau$ (in the $xy$-plane) enclosing two disks of radius $\tau=\kappa/2$. Note also that $x=(0,0,0)$ and $y_\tau$ are the centres of such circles. In addition, recall the cores $\beta_\tau$ are $2\tau$-constrained stadium curves of length $4\tau(\pi+1)$ satisfying $\kappa=2\tau$, and ropelength $2(\pi+1)$. 

We apply Corollary \ref{nonsplit} accordingly to each element in the family by noting that the points $x,y_\tau$ are always in the interior of the base of a cylinder of radius $1$ based on the $xy$-plane since $||x-y_\tau||<2$ to conclude that these unknots are not separable, and therefore gordian. We compute the length of each unlink similarly to Theorem \ref{mingords}. 

\end{proof}


{ \begin{figure} 
 \begin{center}
\includegraphics[width=1\textwidth,angle=0]{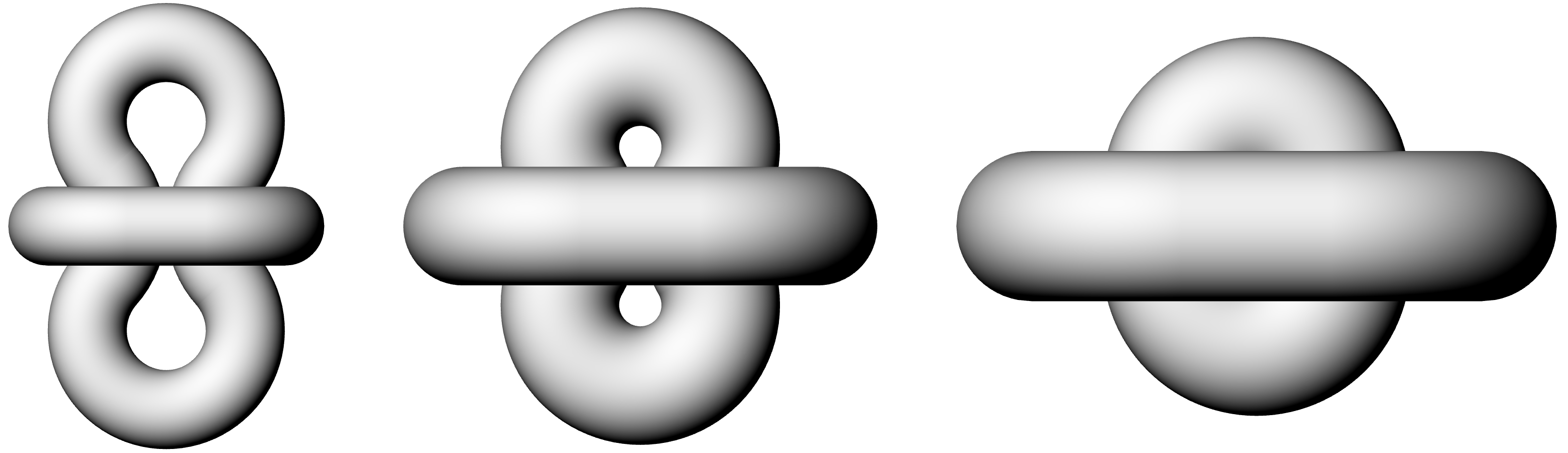}
\end{center}
\caption {Examples gordian unlinks in spaces of thin knots of thickness $1.0, 1.5, 1.9$ respectively and curvature bounded above by 1. These cannot be separated by an isotopy preserving the prescribed constraints.}
\label{figfig5}
\end{figure}}

{ \begin{figure} 
 \begin{center}
\includegraphics[width=.9\textwidth,angle=0]{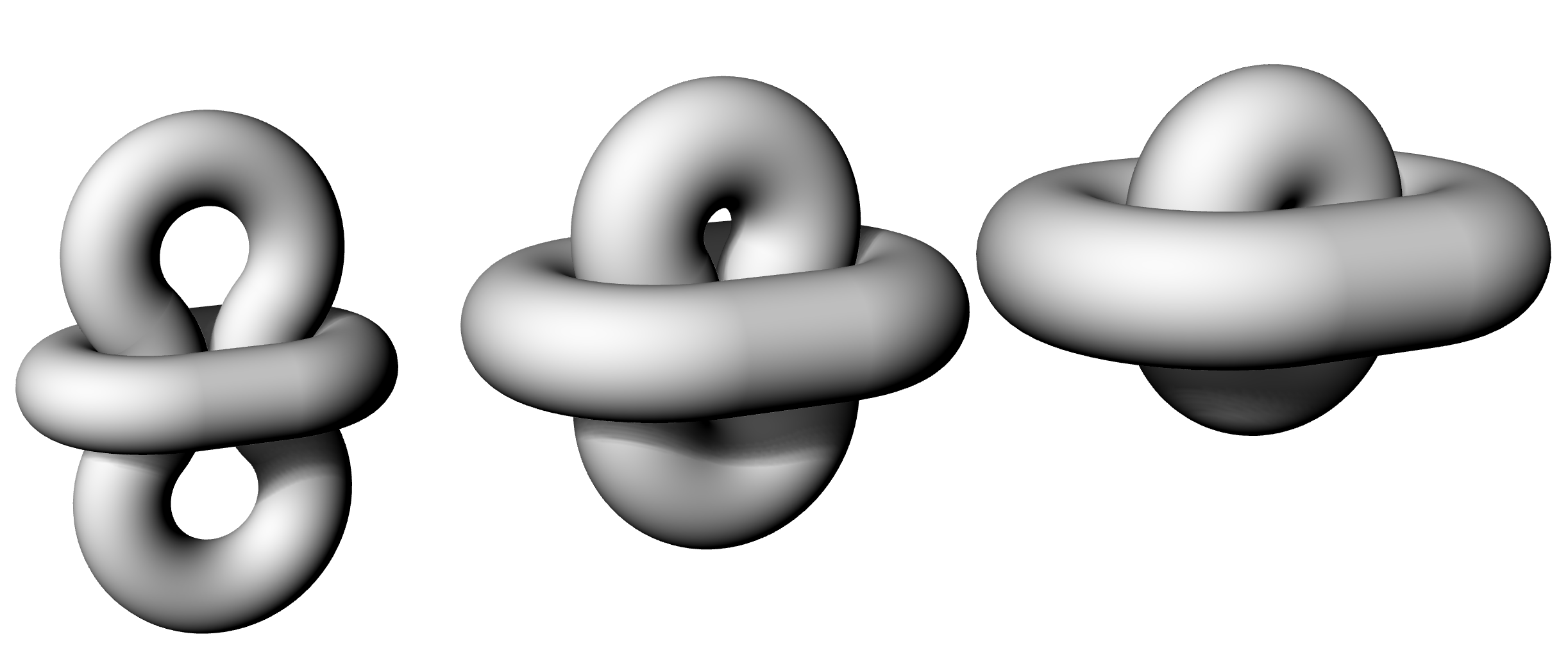}
\end{center}
\caption {
Figure \ref{figfig5} shows pairs of unknots forming a gordian unlink of piecewise constant curvature. The thickness varies across the unlinks in this family.}
\label{figfig6}
\end{figure}}

%
%
%

{\bf Acknowledgments.}
I thank Alejandro Toro and Francisco Tovar-Lopez for their work on the graphics. I am also grateful to J. Hyam Rubinstein for his valuable suggestions and comments.

 





\begin{thebibliography}{20}
\baselineskip=20 true pt
\maketitle \baselineskip=1.15\normalbaselineskip

\bibitem{agol} 
I. Agol, F. C. Marques, and A. Neves, {\em Min-max theory and the energy of links}, J. Amer. Math. Soc., Vol. 29, pp. 561-578, 2016.

\bibitem{ayalahass} 
J. Ayala and J. Hass, {\em Gordian Unlinks}, \url{https://arxiv.org/abs/2005.13168}.

\bibitem{papera} 
J. Ayala, D. Kirszenblat, and J. H. Rubinstein, {\em A geometric approach to shortest bounded curvature paths}, Communications in Analysis and Geometry, Vol. 26, No. 4, 2018.

\bibitem{paperh} 
J. Ayala, D. Kirszenblat, and J. H. Rubinstein, {\em Immersed flat ribbon knots}, \textit{Israel Journal of Mathematics}, accepted for publication, \url{https://arxiv.org/abs/2005.13168}.

\bibitem{paperb} 
J. Ayala, {\em Length minimising bounded curvature paths in homotopy classes}, Topology and its Applications, Vol. 193, pp. 140-151, 2015.

\bibitem{paperc} 
J. Ayala and J. H. Rubinstein, {\em Non-uniqueness of the homotopy class of bounded curvature paths}, (2014), arXiv:1403.4911 [math.MG].

\bibitem{paperd} 
J. Ayala and J. H. Rubinstein, {\em The classification of homotopy classes of bounded curvature paths}, (2014), arXiv:1403.5314v2 [math.MG]. To appear in the \textit{Israel Journal of Mathematics}.

\bibitem{papere} 
J. Ayala, {\em On the topology of the spaces of curvature constrained plane curves}, Advances in Geometry, Vol. 17, No. 3, pp. 283-292, 2017.

\bibitem{paperd2} 
J. Ayala and J. H. Rubinstein, {\em The classification of homotopy classes of bounded curvature paths}, Israel Journal of Mathematics, Vol. 213, No. 1, pp. 79-107, 2016.

\bibitem{hatcher 2} 
T. E. Brendle and A. E. Hatcher, {\em Configuration spaces of rings and wickets}, Commentarii Mathematici Helvetici, Vol. 88, No. 1, pp. 131-162, 2013.

\bibitem{cantarella 1} 
J. Cantarella, R. B. Kusner, and J. M. Sullivan, {\em On the minimum ropelength of knots and links}, Inventiones Mathematicae, Vol. 150, No. 2, pp. 257-286, 2002.

\bibitem{cantarella 2} 
R. B. Kusner and J. M. Sullivan, {\em Criticality for the Gehring link problem}, Geometry \& Topology, Vol. 10, pp. 2055-2116, 2006.

\bibitem{cowardhass} 
A. Coward and J. Hass, {\em Topological and physical knot theory are distinct}, Pacific J. Math., Vol. 276, No. 2, pp. 387-400, 2015.

\bibitem{guardian} 
K. Devlin, {\em Unravelling the myth}, The Guardian, \url{https://www.theguardian.com/science/2001/sep/13/physicalsciences.highereducation}.

\bibitem{diao1} 
Y. Diao, C. Ernst, and E. J. Janse van Rensburg, {\em Thicknesses of knots}, Math. Proc. Cambridge Philos. Soc., Vol. 126, No. 2, pp. 293-310, 1999.

\bibitem{paperf} 
J. Diaz and J. Ayala, {\em Census to bounded curvature paths}, Geometricae Dedicata, Vol. 204, pp. 43-71, 2020.

\bibitem{dubins 1} 
L. E. Dubins, {\em On curves of minimal length with a constraint on average curvature, and with prescribed initial and terminal positions and tangents}, American Journal of Mathematics, Vol. 79, pp. 139-155, 1957.

\bibitem{dubins2} 
L. E. Dubins, {\em On plane curves with curvature}, Pacific J. Math., Vol. 11, No. 2, pp. 471-481, 1961.

\bibitem{durum} 
O. Durumeric, {\em Local structure of ideal knots and shapes}, Topology and its Applications, Vol. 154, No. 17, pp. 3070-3089, 2007.

\bibitem{hyde} 
M. E. Evans, V. Robins, and S. T. Hyde, {\em Ideal geometry of periodic entanglements}, Proceedings of the Royal Society A, Vol. 471, No. 2181, pp. 1-23, 2015.

\bibitem{freedman} 
M. Freedman, Z. He, and Z. Wang, {\em Möbius energy of knots and unknots}, Annals of Mathematics, Second Series, Vol. 139, No. 1, pp. 1-50, 1994.

\bibitem{gonzalez} 
O. Gonzales and J. Maddocks, {\em Global curvature, thickness, and the ideal shapes of knots}, Proc. Natl. Acad. Sci. USA, Vol. 96, No. 9, pp. 4769-4773, 1999.

\bibitem{kaitrich} 
V. Katritch, J. Bednar, D. Michoud, R. G. Scharein, J. Dubochet, and A. Stasiak, {\em Nature}, Vol. 384, pp. 142, 1996.

\bibitem{millet} 
K. C. Millett and E. J. Rawdon, {\em Energy, ropelength, and other physical aspects of equilateral knots}, Journal of Computational Physics, Vol. 186, pp. 426-456, 2003.

\bibitem{hatcher 1} 
A. Hatcher, {\em A proof of the Smale conjecture, ${\displaystyle \scriptstyle {\mathrm {Diff} }(S^{3})\simeq {\mathrm {O} }(4)}$}, Annals of Mathematics, Vol. 117, No. 3, pp. 553-607, 1983.

\bibitem{wkusner} 
R. Kusner and W. Kusner, {\em A gordian pair of links}, arXiv:1908.05610v1 [math.GT], 2019.

\bibitem{simon1} 
R. A. Litherland, J. Simon, O. Durumeric, and E. Rawdon, {\em Thickness of knots}, Topology Appl., Vol. 91, No. 3, pp. 233-244, 1999.

\bibitem{markov}  
A. A. Markov, {\em Some examples of the solution of a special kind of problem on greatest and least quantities}, Soobshch. Karkovsk. Mat. Obshch., Vol. 1, pp. 250-276, 1887.

\bibitem{ohara} 
J. O'Hara, {\em Energy of a knot}, Topology, Vol. 30, No. 2, pp. 241-247, 1991.

\bibitem{heiko} 
P. Strzelecki and H. von der Mosel, {\em How averaged Menger curvatures control regularity and topology of curves and surfaces}, Journal of Physics: Conference Series, Vol. 544, pp. 012018, 2014.

\bibitem{pieransky1} 
P. Pieranski, S. Przybyl, and A. Stasiak, {\em 16th IMACS World Congress on Scientific Computation, Applied Mathematics and Simulation}, Lausanne, Switzerland, August 21-25, 2000.

\bibitem{pieransky2} 
P. Pieranski, S. Przybyl, and A. Stasiak, {\em gordian unknots}, arXiv preprint physics/0103080, 2001.

\bibitem{sussman} 
H. J. Sussmann, {\em Shortest 3-dimensional paths with a prescribed curvature bound}, Proceedings of the 34th IEEE Conference on Decision and Control, pp. 3306-3312, 1995.

 
 \end{thebibliography}
 \end{document}